\documentclass[12pt]{amsart}

\usepackage{amsmath}
\usepackage{amssymb}
\usepackage{amsfonts}
\usepackage{amsthm}
\usepackage{enumerate}
\usepackage{hyperref}
\usepackage{color}

\theoremstyle{plain}
\newtheorem{thm}{Theorem}[section]

\newtheorem{prop}[thm]{Proposition}
\newtheorem{lem}[thm]{Lemma}
\newtheorem{cor}[thm]{Corollary}

\newtheorem{conj}[thm]{Conjecture}

\theoremstyle{definition}
\newtheorem{dfn}[thm]{Definition}

\newtheorem{rem}[thm]{Remark}

\newtheorem{dfns-rems}[thm]{Definitions and Remarks}
\newtheorem{notas-rems}[thm]{Notations and Remarks}
\newtheorem{exmps-rems}[thm]{Examples and Remarks}

\begin{document}


\title[Sdepth of weakly polymatroidal ideals and squarefree monomial ideals]{Stanley depth of weakly polymatroidal ideals and squarefree monomial ideals}


\author[S. A. Seyed Fakhari]{S. A. Seyed Fakhari}

\address{S. A. Seyed Fakhari, Department of Mathematical Sciences,
Sharif University of Technology, P.O. Box 11155-9415, Tehran, Iran.}

\email{fakhari@ipm.ir}

\urladdr{http://math.ipm.ac.ir/fakhari/}


\begin{abstract}
Let $I$ be a weakly polymatroidal ideal or a squarefree monomial ideal of a polynomial ring $S$. In this paper we provide  a lower bound for the Stanley depth of $I$ and $S/I$. In particular we prove that if $I$ is a squarefree monomial ideal which is generated in a single degree, then  ${\rm sdepth}(I)\geq n-\ell(I)+1$ and ${\rm sdepth}(S/I)\geq n-\ell(I)$, where $\ell(I)$ denotes the analytic spread of $I$. This proves a conjecture of the author in a special case.
\end{abstract}


\subjclass[2000]{Primary: 13C15, 05E99; Secondary: 13C13}


\keywords{Monomial ideal, Depth, Stanley depth, Analytic spread, Weakly polymatroidal ideal, Squarefree monomial ideal, Rank, Affine rank}


\maketitle


\section{Introduction} \label{sec1}

Let $\mathbb{K}$ be a field and let $S=\mathbb{K}[x_1,\dots,x_n]$ be the
polynomial ring in $n$ variables over $\mathbb{K}$. Let $M$ be a
finitely generated $\mathbb{Z}^n$-graded $S$-module. Let $u\in M$ be a
homogeneous element and $Z\subseteq \{x_1,\dots,x_n\}$. The $\mathbb
{K}$-subspace $u\mathbb{K}[Z]$ generated by all elements $uv$ with $v\in
\mathbb{K}[Z]$ is called a {\it Stanley space} of dimension $|Z|$, if it is
a free $\mathbb{K}[Z]$-module. Here, as usual, $|Z|$ denotes the
number of elements of $Z$. A decomposition $\mathcal{D}$ of $M$ as a finite
direct sum of Stanley spaces is called a {\it Stanley decomposition} of
$M$. The minimum dimension of a Stanley space in $\mathcal{D}$ is called the
{\it Stanley depth} of $\mathcal{D}$ and is denoted by ${\rm sdepth}
(\mathcal {D})$. The quantity $${\rm sdepth}(M):=\max\big\{{\rm sdepth}
(\mathcal{D})\mid \mathcal{D}\ {\rm is\ a\ Stanley\ decomposition\ of}\
M\big\}$$ is called the {\it Stanley depth} of $M$. Stanley \cite{s}
conjectured that $${\rm depth}(M) \leq {\rm sdepth}(M)$$ for every
$\mathbb{Z}^n$-graded $S$-module $M$. For a reader friendly introduction
to Stanley depth, we refer to \cite{psty}.

Let $I$ be a monomial ideal of $S$ with Rees algebra $\mathcal{R}(I)$
and let $\mathfrak{m}=(x_1,\ldots,x_n)$ be the graded maximal ideal of $S$.
Then the $\mathbb{K}$-algebra $\mathcal{R}(I)/\mathfrak{m}\mathcal{R}(I)$
is called the {\it fibre ring} and its Krull dimension is called the {\it
analytic spread} of $I$, denoted by $\ell(I)$. This invariant is a measure
for the growth of the number of generators of the powers of $I$. Indeed,
for $k\gg 0$, the Hilbert function $H(\mathcal{R}(I)/\mathfrak{m}
\mathcal{R}(I),\mathbb{K},k)={\rm dim}_\mathbb {K}(I^k/\mathfrak{m}I^k)$,
which counts the number of generators of the powers of $I$, is a polynomial
function of degree $\ell(I)-1$.

In this paper we consider some linear algebraic approximations of the analytic spread of a monomial ideal. Indeed, assume that $v_1, \ldots, v_t$ are $t$ vectors in $\mathbb{Q}^n$. Then they are called to be {\it linearly dependent} if there exist rational numbers $c_1, \ldots, c_t$, not all zero, for which $$c_1v_1+\ldots +c_tv_t=0.$$Similarly they are {\it affinely dependent}, if in addition the sum of coefficients is zero: $$\sum_{i=1}^tc_i=0.$$If $v_1, \ldots, v_t$ are not linearly dependent (resp. affinely dependent), then they are said to be {\it linaerly independent} (resp. {\it affinely independent}). Now we associate two invariants to every monomial ideal $I$, which are called the rank and the affine rank of $I$. For every vector $\mathbf{a}=(a_1, \ldots, a_n)$ of non-negative integers, we denote the monomial $x_1^{a_1}\ldots x_n^{a_n}$ by $\mathbf{x^{a}}$.
\begin{dfn} \label{ar}
Let $I\subseteq S=\mathbb{K}[x_1,\ldots,x_n]$ be a monomial ideal and $G(I)=\{\mathbf{x^{a_1}}, \ldots, \mathbf{x^{a_m}}\}$ be the set of minimal monomial generators of $I$. The {\it rank} of $I$, denoted by ${\rm rank}(I)$ is the cardinality of the largest linearly independent subset of $\{\mathbf{a_1}, \ldots, \mathbf{a_m}\}$. Similarly the {\it affine rank} of $I$, denoted by ${\rm arank}(I)$ is the cardinality of the largest affinely independent subset of $\{\mathbf{a_1}, \ldots, \mathbf{a_m}\}$.
\end{dfn}

It is clear from Definition \ref{ar} that for every monomial ideal $I$,  the inequality ${\rm arank}(I)\geq {\rm rank}(I)$ holds. It is known \cite[Lemma
10.3.19]{hh'} that if $I$ is a monomial ideal which is generated in a single degree, then $\ell(I)={\rm rank}(I)$. The following Proposition shows that in this case we also have $\ell(I)={\rm arank}(I)$.

\begin{prop} \label{single}
Let $I$ be a monomial ideal, which is generated in a single degree. Then  $\ell(I)={\rm rank}(I)={\rm arank}(I)$.
\end{prop}

\begin{proof}
It is sufficient to prove the second equality. Assume that ${\rm arank}(I)=t$. Therefore, there exist integers $1\leq i_1 < \ldots < i_t\leq m$ such that the equalities $$c_1\mathbf{a_{i_1}}+ \ldots +c_t\mathbf{a_{i_t}}=0$$ and $$c_1+ \ldots + c_t=0,$$with $c_i \in \mathbb{Q}$, for every $1\leq i \leq t$, imply that $c_1= \ldots = c_t=0$. Since $I$ is generated in a single degree, $\mathbf{a_{i_1}}, \ldots ,\mathbf{a_{i_t}}$ are linearly independent over $\mathbb{Q}$. Indeed, assume that there exist rational numbers $d_1, \ldots, d_t$ such that $$d_1\mathbf{a_{i_1}}+ \ldots +d_t\mathbf{a_{i_t}}=0.$$Now for every $1\leq j \leq t$, the sum of the components of $\mathbf{a_{i_j}}$ is equal to $k$ and thus, the sum of the components of $$d_1\mathbf{a_{i_1}}+ \ldots +d_t\mathbf{a_{i_t}}$$ is equal to $$d_1k+ \ldots+ d_tk$$and this shows that $$d_1+ \ldots + d_t=0.$$Therefore $$d_1= \ldots = d_t=0.$$Hence $\mathbf{a_{i_1}}, \ldots ,\mathbf{a_{i_t}}$ are linearly independent over $\mathbb{Q}$. Therefore, ${\rm rank}(I)\geq t$. Since we always have ${\rm arank}(I)\geq {\rm rank}(I)$, it follows that ${\rm arank}(I)= {\rm rank}(I)$.
\end{proof}

In \cite{psy}, the authors prove that if $I\subset S$ is a weakly polymatroidal ideal $I$ (see Definition \ref{weak}), which is generated in a single degree, then  ${\rm depth}(S/I)\geq n-\ell(I)$, ${\rm sdepth}(S/I)\geq n-\ell(I)$ and ${\rm sdepth}(I)\geq n-\ell(I)+1$. In Section \ref{sec2} we generalize this result by proving that for every weakly polymatroidal ideal $I$, the inequalities $${\rm sdepth}(I)\geq n-{\rm arank}(I)+1, \ \ \ {\rm sdepth}(S/I)\geq n-{\rm arank}(I)$$ and $${\rm depth}(S/I)\geq n-{\rm arank}(I)$$hold (see Theorem \ref{main}).\\
In \cite{s1}, the author conjectures that for every integrally closed monomial ideal, the inequalities ${\rm sdepth}(S/I)\geq n-\ell(I)$ and ${\rm sdepth}(I)\geq n-\ell(I)+1$ hold (see Conjecture \ref{conje}). In Section \ref{sec3}, we prove this conjecture for every squarefree monomial ideal which is generated in a single degree. In fact, we prove some stronger result. We show that for every squarefree monomial ideal $I$ of the polynomial ring $S$, the inequalities $${\rm sdepth}(I)\geq n-{\rm rank}(I)+1$$ and $${\rm sdepth}(S/I)\geq n-{\rm rank}(I)$$ hold (see Theorem \ref{smain}).


\section{Stanley depth of weakly polymatroidal ideals} \label{sec2}

Weakly polymatroidal ideals are generalization of polymatroidal ideals and
they are defined as follows.

\begin{dfn} [\cite{hh'}, Definition 12.7.1] \label{weak}
A monomial ideal $I$ of $S=\mathbb{K}[x_1,\dots,x_n]$ is called {\it weakly
polymatroidal} if for every two monomials $u = x_1^{a_1} \ldots x_n^{a_n}$
and $v = x_1^{b_1} \ldots x_n^{b_n}$ in $G(I)$ such that $a_1 = b_1,
\ldots, a_{t-1} = b_{t-1}$ and $a_t > b_t$ for some $t$, there exists $j >
t$ such that $x_t(v/x_j)\in I$.
\end{dfn}

The aim of this section is to provide a lower bound for the depth and the Stanley depth of weakly polymatroidal ideals. As usual for every monomial $u$, the {\it support} of $u$, denoted by ${\rm Supp}(u)$, is the set of variables, which divide $u$.

\begin{lem} \label{colon}
Let $I$ be a weakly polymatroidal ideal and let $G(I)=\{u_1, \ldots, u_m\}$ be the set of minimal monomial generators of $I$. Assume that $$x_1\in \bigcup_{i=1}^m{\rm Supp}(u_i).$$ Then $(I:x_1)$ is a weakly polymatroidal ideal which is minimally generated by the set $$\mathcal{G}=\{\frac{u_i}{x_1}| \ u_i\in G(I) \ {\rm and} \ x_1 \ {\rm divides} \ u_i \}.$$
\end{lem}

\begin{proof}
It is clear that the ideal generated by $\mathcal{G}$ is a weakly polymatroidal
ideal. Thus, we prove that $(I:x_1)$ is generated by the set $\mathcal{G}$. Without loss of generality, we may
assume that $u_1, \ldots, u_t$ are divisible by $x_1$ and $u_{t+1},\ldots,
u_m$ are not divisible by $x_1$, where $1\leq t\leq m$. Let $v_i=u_i/x_1$
($1\leq i\leq t$). We should prove that $(I:x_1)$ is generated by $v_1,\ldots,
v_t$. Let $v\in(I:x_1)$ be a monomial. Then
$x_1v\in I$ and so there exists $1\leq i\leq m$ in such a way that $u_i$
divides $x_1v$. If $1\leq i \leq t$, then $v$ is divisible by $v_i$ and
therefore, $v\in (v_1,\ldots,v_t)$. Hence, we may assume that $i\geq
t+1$. Now $u_i$ is not divisible by $x_1$ and thus $u_i|v$. Since $$x_1\in \bigcup_{i=1}^m{\rm Supp}(u_i),$$ Definition
\ref{weak} implies that there exists $j\geq 2$ such that $x_1u_i/x_j \in I$. Hence, there exists
$1\leq s \leq m$, such that $u_s$ divides $x_1u_i/x_j$. If $t+1\leq s \leq m$, then $u_s$ divides $u_i/x_j$ and thus $u_s$ properly divides $u_i$, which is a contradiction, because $G(I)$ is the set of minimal monomial generators of $I$. It follows that $1 \leq s \leq t$. Therefore, $v_s$ divides $u_i/x_j$ and hence, it divides $u_i$. Since $v$ is divisible by $u_i$, we conclude that $v_s$ divides $v$. This shows that $v\in
(v_1, \ldots v_t)$ and completes the proof of the
lemma.
\end{proof}

The following lemma shows that the affine rank of a weakly polymatroidal ideal does not increase under the colon operation with respect to the variable $x_1$

\begin{lem} \label{wcolon}
Let $I$ be a weakly polymatroidal ideal. Then ${\rm arank}((I:x_1))\leq {\rm arank}(I)$.
\end{lem}
\begin{proof}
If $I=(I:x_1)$, then there is nothing to prove. So assume that $I\neq(I:x_1)$.
Let $G(I)=\{u_1, \ldots, u_m\}$ be the set of minimal monomial generators of $I$. Since $I\neq(I:x_1)$, it follows that $$x_1\in \bigcup_{i=1}^m{\rm Supp}(u_i).$$ Without loss of generality, we may
assume that $u_1, \ldots, u_t$ are divisible by $x_1$ and $u_{t+1},\ldots,
u_m$ are not divisible by $x_1$, where $1\leq t\leq m$. Let $v_i=u_i/x_1$
($1\leq i\leq t$). By Lemma \ref{colon}, the set $\{v_1, \ldots, v_t\}$ is the set of minimal monomial generators of $(I:x_1)$. For simplicity we assume that $\mathbf{a_i}$ is the exponent vector of $v_i$, ($1\leq i \leq t$). Suppose that ${\rm arank}((I:x_1))=s$ and choose the monomials $v_{j_1}, \ldots, v_{j_s}$, such that the equalities $$c_1\mathbf{a_{j_1}}+ \ldots +c_s\mathbf{a_{j_s}}=0$$ and $$c_1+ \ldots + c_s=0,$$with $c_i \in \mathbb{Q}$, for every $1\leq i \leq s$, imply that $c_1= \ldots = c_s=0$. Note that for every $1\leq i \leq t$, the exponent vector of $u_i$ is equal to $\mathbf{a_i}+ \mathbf{e_1}$, where $\mathbf{e_1}$ is the first vector in the standard basis of $\mathbb{Q}^n$. Now assume that there exist $d_1, \ldots, d_s\in \mathbb{Q}$, such that $d_1+\ldots+d_s=0$ and $$d_1(\mathbf{a_{j_1}}+ \mathbf{e_1})+ \ldots +d_s(\mathbf{a_{j_s}}+\mathbf{e_1})=0.$$
Therefore $$d_1\mathbf{a_{j_1}}+ \ldots +d_s\mathbf{a_{j_s}}+(d_1+ \ldots + d_s)\mathbf{e_1}=0.$$Since $d_1+\ldots+d_s=0$, it follows that $$d_1\mathbf{a_{j_1}}+ \ldots +d_s\mathbf{a_{j_s}}.$$By the choice of $v_{j_1}, \ldots, v_{j_s}$,  we conclude that $d_1= \ldots =d_s=0$. Thus,
${\rm arank}(I)\geq s$ and this proves our assertion.
\end{proof}

In the following lemma we consider the behavior of the affine rank of an arbitrary monomial ideal under the elimination of $x_1$.

\begin{lem} \label{del}
Let $I$ be a monomial ideal of $S=\mathbb{K}[x_1,\ldots,x_n]$, such that $$x_1\in \bigcup_{u\in G(I)}{\rm Supp}(u).$$Let $S'=\mathbb{K}[x_2, \ldots, x_n]$ be the polynomial ring obtained from $S$ by deleting the variable $x_1$ and consider the ideal $I'=I\cap S'$. Then ${\rm arank}(I')+1\leq {\rm arank}(I)$.
\end{lem}

\begin{proof}
Let $G(I)=\{u_1, \ldots, u_m\}$ be the set of minimal monomial generators of $I$. For simplicity we assume that $\mathbf{a_i}$ is the exponent vector of $u_i$,  ($1\leq i \leq m$). Without loss of generality, we may
assume that $u_1, \ldots, u_t$ are divisible by $x_1$ and $u_{t+1},\ldots,
u_m$ are not divisible by $x_1$, where $1\leq t\leq m$. Then the set $\{u_{t+1},\ldots,
u_m\}$ is the set of minimal monomial generators of $I'$. Assume that ${\rm arank}(I')=s$. Thus, there exist integers $t+1\leq j_1 < j_2 < \ldots < j_s\leq m$, such that the equalities $$c_1\mathbf{a_{j_1}}+ \ldots +c_s\mathbf{a_{j_s}}=0$$ and $$c_1+ \ldots + c_s=0,$$with $c_i \in \mathbb{Q}$, for every $1\leq i \leq s$, imply that $c_1= \ldots = c_s=0$. Now we consider the set $\{u_1, u_{j_1}, \ldots, u_{j_s}\}$ and assume that there exist $d_0, d_1, \ldots, d_s\in \mathbb{Q}$,
such that $d_0+d_1+\ldots+d_s=0$ and $$d_0\mathbf{a_1}+d_1\mathbf{a_{j_1}}+ \ldots +d_s\mathbf{a_{j_s}}=0.$$Looking at the first component of the vector
$d_0\mathbf{a_1}+d_1\mathbf{a_{j_1}}+ \ldots +d_s\mathbf{a_{j_s}}$, it follows that $d_0=0$ and hence, $d_1+\ldots+d_s=0$ and $$d_1\mathbf{a_{j_1}}+ \ldots +d_s\mathbf{a_{j_s}}=0.$$By the choice of integers $j_1, \ldots, j_s$, we conclude that $d_1= \ldots = d_s=0$. Therefore ${\rm arank}(I)\geq s+1={\rm arank}(I')+1$.
\end{proof}

\begin{rem}
It is completely clear from the proof of the Lemma \ref{del}, that one can consider any arbitrary variable instead of $x_1$.
\end{rem}

We are now ready to state and prove the main result of this section.

\begin{thm} \label{main}
Let $I$ be a weakly polymatroidal ideal of $S=\mathbb{K}[x_1,\ldots,x_n]$. Then we have the following
assertions:

\begin{itemize}
\item[(i)] ${\rm sdepth}(I)\geq n-{\rm arank}(I)+1$ and ${\rm sdepth}(S/I)\geq
    n-{\rm arank}(I)$.\\[-0.3cm]
\item[(ii)] ${\rm depth}(S/I)\geq n-{\rm arank}(I)$.
\end{itemize}
\end{thm}

\begin{proof}
We prove (i) and (ii) simultaneously by induction on $n$ and $$\sum_{u\in G(I)} {\rm deg} (u),$$ where $G(I)$ is the set of minimal monomial generators of $I$. If
$n=1$ or $$\sum_{u\in G(I)} {\rm deg} (u)=1,$$ then $I$ is a principal ideal and so we have ${\rm arank}(I)=1$, ${\rm
sdepth}(I)=n$, ${\rm depth}(S/I)=n-1$ and by \cite[Theorem
1.1]{r}, ${\rm sdepth}(S/I)=n-1$. Therefore, in
these cases, the inequalities in (i) and (ii) are trivial.

We now assume that $n\geq 2$ and
$$\sum_{u\in G(I)} {\rm deg} (u)\geq 2.$$Let $S'=\mathbb{K}[x_2, \ldots, x_n]$ be the polynomial ring obtained from $S$ by deleting the variable $x_1$ and consider the ideals $I'=I\cap S'$ and
$I''=(I:x_1)$. If $$x_1\notin \bigcup_{u\in G(I)}{\rm Supp}(u_i),$$then the induction hypothesis on $n$ implies that
$${\rm depth}(S/I)={\rm depth}(S'/I')+1\geq (n-1)-{\rm arank}(I')+1=n-{\rm arank}(I).$$On the other hand, by \cite[Theorem
1.1]{r} and \cite[Lemma 3.6]{hvz}, we conclude that ${\rm sdepth}(S/I)={\rm
sdepth}(S'/I')+1$ and ${\rm sdepth}(I)={\rm sdepth} (I')+1$.
Therefore, using the induction hypothesis on $n$ we conclude that ${\rm
sdepth}(I)\geq n-{\rm arank}(I)+1$ and ${\rm sdepth}(S/I)\geq n-{\rm arank}(I)$.
Therefore, we may assume that $$x_1\in \bigcup_{u\in G(I)}{\rm Supp}(u_i),$$ Now $I=I'S'\oplus x_1I''S$ and $S/I=(S'/I'S')\oplus
x_1(S/I''S)$ and therefore by definition of  the Stanley depth we have
\[
\begin{array}{rl}
{\rm sdepth}(I)\geq \min \{{\rm sdepth}_{S'}(I'S'), {\rm sdepth}_S(I'')\},
\end{array} \tag{1} \label{1}
\]
and
\[
\begin{array}{rl}
{\rm sdepth}(S/I)\geq \min \{{\rm sdepth}_{S'}(S'/I'S'), {\rm sdepth}_S(S/I'')\}.
\end{array} \tag{2} \label{2}
\]
On the other hand, by applying the depth lemma on the exact sequence
\[
\begin{array}{rl}
0\longrightarrow S/(I:x_1)\longrightarrow S/I\longrightarrow S/(I, x_1)
\longrightarrow 0
\end{array}
\]
we conclude that
\[
\begin{array}{rl}
{\rm depth}(S/I)\geq \min \{{\rm depth}_{S'}(S'/I'S'), {\rm depth}_S(S/I'')\}.
\end{array} \tag{3} \label{3}
\]
Using Lemma \ref{colon} it follows that $I''$ is a weakly polymatroidal ideal and by Lemma \ref{wcolon} we conclude that that ${\rm arank}(I'')\leq {\rm arank}(I)$. Hence our induction hypothesis on $$\sum_{u\in G(I)} {\rm deg}(u)$$
 implies that $${\rm depth}_S(S/I'')\geq n-{\rm arank}(I'')\geq n-{\rm arank}(I),$$ $${\rm sdepth}_S(S/I'')\geq n-{\rm arank}(I'')\geq n-{\rm arank}(I),$$ and $${\rm sdepth}_S(I'')\geq n-{\rm arank}(I'')+1\geq n-{\rm arank}(I)+1.$$

On the other hand $I'S'$ is a weakly polymatroidal ideal and since $$x_1 \in \bigcup_{i=1}^s {\rm Supp}(u_i),$$using Lemma \ref{del} we conclude that ${\rm arank}(I'S')\leq {\rm arank}(I)-1$ and therefore by
our induction hypothesis on $n$ we conclude that

$${\rm sdepth}_{S'}(I'S')\geq (n-1)-{\rm arank}(I'S')+1\geq (n-1)-({\rm arank}(I)-1)+1$$ $$=n-{\rm arank}(I)+1,$$

and similarly ${\rm sdepth}_{S'}(S'/I'S')\geq n-{\rm arank}(I)$ and ${\rm
depth}_{S'}(S'/I'S')\geq n-{\rm arank}(I)$. Now the assertions follow by inequalities (\ref{1}), (\ref{2})
and (\ref{3}).
\end{proof}

As an immediate consequence of Proposition \ref{single} and Theorem \ref{main}, we conclude the following result which appeared in \cite{psy}.

\begin{cor}
Let $I$ be a weakly polymatroidal ideal of $S=\mathbb{K}[x_1,\ldots,x_n]$
which is generated in a single degree. Then we have the following
assertions:

\begin{itemize}
\item[(i)] ${\rm sdepth}(I)\geq n-\ell(I)+1$ and ${\rm sdepth}(S/I)\geq
    n-\ell(I)$.\\[-0.3cm]
\item[(ii)] ${\rm depth}(S/I)\geq n-\ell(I)$.
\end{itemize}
\end{cor}

Using Theorem \ref{main} we provide an upper bound for the height of associated primes of a weakly polymatroidal ideal.

\begin{cor}
Let $I$ be a weakly polymatroidal ideal of $S=\mathbb{K}[x_1, \ldots,x_n]$. Then
$$\max\{{\rm ht}(\frak{p})\mid \frak{p}\in {\rm Ass}(S/I)\}\leq {\rm arank}(I).$$
\end{cor}

\begin{proof}
Let $\frak{p}\in {\rm Ass}(S/I)$ be given. By \cite[Proposition 1.2.13]{bh}
we have ${\rm depth}(S/I)\leq n-{\rm ht}(\frak{p})$, while by Theorem
\ref{main} we have ${\rm depth}(S/I)\geq n-{\rm arank}(I)$. This implies that
${\rm ht}(\frak{p})\leq {\rm arank}(I)$ for every $\frak{p}\in {\rm Ass}(S/I)$ and
completes the proof of the corollary.
\end{proof}


\section{Stanley depth of squarefree monomial ideals} \label{sec3}

Let $I\subset S$ be an arbitrary ideal. An element $f \in S$ is
{\it integral} over $I$, if there exists an equation
$$f^k + c_1f^{k-1}+ \ldots + c_{k-1}f + c_k = 0 {\rm \ \ \ \ with} \ c_i\in I^i.$$
The set of elements $\overline{I}$ in $S$ which are integral over $I$ is the {\it integral closure}
of $I$. It is known that the integral closure of a monomial ideal $I\subset S$ is a monomial ideal
generated by all monomials $u \in S$ for which there exists an integer $k$ such that
$u^k\in I^k$ (see \cite[Theorem 1.4.2]{hh'}).

In \cite{s1}, the author proposed the following conjecture regarding the Stanley depth of integrally closed monomial ideals.

\begin{conj} \label{conje}
Let $I\subset S$ be an integrally closed monomial ideal. Then ${\rm sdepth}(S/I)\geq n-\ell(I)$ and ${\rm sdepth} (I)\geq n-\ell(I)+1$.
\end{conj}

In this section we prove that Conjecture \ref{conje} is true for every squarefree monomial ideal which is generated in a single degree. Indeed we show that for every squarefree monomial ideal $I$ of the polynomial ring $S$, the inequalities ${\rm sdepth}(I)\geq n-{\rm rank}(I)+1$ and ${\rm sdepth}(S/I)\geq n-{\rm rank}(I)$ hold (see Theorem \ref{smain}).

First we need the following lemma.

\begin{lem} \label{swcolon}
Let $I$ be a squarefree monomial ideal. Then for every $1\leq j \leq  n$ we have ${\rm rank}((I:x_j))\leq {\rm rank}(I)$.
\end{lem}
\begin{proof}
Let $G(I)=\{u_1, \ldots, u_m\}$ be the set of minimal monomial generators of $I$. Without loss of generality, we may
assume that $u_1, \ldots, u_t$ are divisible by $x_j$ and $u_{t+1},\ldots,
u_m$ are not divisible by $x_j$, where $0\leq t\leq m$. Put $v_i=u_i/x_j$, if $1\leq i \leq t$ and $v_i=u_i$, if $t+1\leq i \leq m$.  For simplicity we assume that $\mathbf{a_i}$ is the exponent vector of $u_i$ and $\mathbf{b_i}$ is the exponent vector of $v_i$  ($1\leq i \leq m$). To prove the assertion one just note that for every $k\neq j$ and every $1\leq i \leq m$, the $k$th component of $\mathbf{a_i}$  and $\mathbf{b_i}$ are the same and for $k=j$, the $k$th component of $\mathbf{b_i}$ is always zero.
\end{proof}

We are now ready to state and prove the main result of this section.

\begin{thm} \label{smain}
Let $I$ be a squarefree monomial ideal of $S=\mathbb{K}[x_1,\ldots,x_n]$. Then ${\rm sdepth}(I)\geq n-{\rm rank}(I)+1$ and ${\rm sdepth}(S/I)\geq n-{\rm rank}(I)$.
\end{thm}

\begin{proof}
Let $G(I)$ be the set of minimal monomial generators of $I$. We prove the assertions by induction on $n$. If
$n=1$ then $I$ is a principal ideal and so we have ${\rm rank}(I)=1$, ${\rm
sdepth}(I)=n$ and by \cite[Theorem
1.1]{r}, ${\rm sdepth}(S/I)=n-1$. Therefore, in
this case, there is nothing to prove.

We now assume that $n\geq 2$. Let $S'=\mathbb{K}[x_2, \ldots, x_n]$ be the polynomial ring obtained from $S$ by deleting the variable $x_1$ and consider the ideals $I'=I\cap S'$ and
$I''=(I:x_1)$. If $$x_1\notin \bigcup_{u\in G(I)}{\rm Supp}(u_i),$$then by \cite[Theorem
1.1]{r} and \cite[Lemma 3.6]{hvz}, we conclude that ${\rm sdepth}(S/I)={\rm
sdepth}(S'/I')+1$ and ${\rm sdepth}(I)={\rm sdepth} (I')+1$.
Therefore, using our induction hypothesis, we conclude that ${\rm
sdepth}(I)\geq n-{\rm rank}(I)+1$ and ${\rm sdepth}(S/I)\geq n-{\rm rank}(I)$.
Hence we may assume that $$x_1\in \bigcup_{u\in G(I)}{\rm Supp}(u_i),$$ Now $I=I'S'\oplus x_1I''S$ and $S/I=(S'/I'S')\oplus
x_1(S/I''S)$ and therefore by the definition of Stanley depth we have
\[
\begin{array}{rl}
{\rm sdepth}(I)\geq \min \{{\rm sdepth}_{S'}(I'S'), {\rm sdepth}_S(I'')\},
\end{array} \tag{1} \label{1}
\]
and
\[
\begin{array}{rl}
{\rm sdepth}(S/I)\geq \min \{{\rm sdepth}_{S'}(S'/I'S'), {\rm sdepth}_S(S/I'')\}.
\end{array} \tag{2} \label{2}
\]
Note that the generators of $I''$ belong to $S'$. Therefore our induction hypothesis implies that $${\rm sdepth}_{S'}(S'/I'')\geq (n-1)-{\rm rank}(I'')$$ and  $${\rm sdepth}_{S'}(S'/I'')\geq (n-1)-{\rm rank}(I'')+1$$ Using Lemma \ref{swcolon} together with \cite[Theorem
1.1]{r} and \cite[Lemma 3.6]{hvz}, we conclude that
$${\rm sdepth}(S/I'')={\rm sdepth}_{S'}(S'/I'')+1\geq (n-1)-{\rm rank}(I'')+1\geq n-{\rm rank}(I),$$
and
$${\rm sdepth}_S(I'')={\rm sdepth}_{S'}(I'')+1\geq (n-1)-{\rm rank}(I'')+1+1\geq n-{\rm rank}(I)+1.$$

On the other hand, since $$x_1 \in \bigcup_{i=1}^s {\rm Supp}(u_i),$$ it follows that ${\rm rank}(I'S')\leq {\rm rank}(I)-1$ and therefore by
our induction hypothesis we conclude that

$${\rm sdepth}_{S'}(I'S')\geq (n-1)-{\rm rank}(I'S')+1\geq (n-1)-({\rm rank}(I)-1)+1$$ $$=n-{\rm rank}(I)+1,$$
and similarly ${\rm sdepth}_{S'}(S'/I'S')\geq n-{\rm rank}(I)$. Now the assertions follow by inequalities (\ref{1}) and (\ref{2}).
\end{proof}
As an immediate consequence of Proposition \ref{single} and Theorem \ref{smain} we conclude that Conjecture \ref{conje} is true for every squarefree monomial ideal which is generated in a single degree.

\begin{cor} \label{sconj}
Let $I$ be a squarefree monomial ideal of $S=\mathbb{K}[x_1,\ldots,x_n]$ which is generated in a single degree. Then ${\rm sdepth}(I)\geq n-\ell(I)+1$ and ${\rm sdepth}(S/I)\geq n-\ell(I)$.
\end{cor}


\section*{Acknowledgment}
The author thanks Volkmar Welker, Siamak Yassemi and Rahim Zaare-Nahandi for reading an earlier version of this article and for their helpful comments.



\end{document}